\documentclass[12pt,reqno]{amsart}
\usepackage{amsmath,amsthm,amssymb,amsfonts,amscd}
\usepackage{mathrsfs}
\usepackage{bbm}
\usepackage{bbding}
\usepackage[backref=page]{hyperref}
\usepackage{geometry}
\usepackage{color}
\usepackage{xcolor}

\usepackage{picture,epic}
\usepackage{tikz}

\usepackage{etoolbox}
\usepackage{orcidlink}

\newif\ifrevisionmode
\revisionmodetrue

\numberwithin{equation}{section}

\theoremstyle{plain}
\newtheorem{theorem}{Theorem}[section]

\theoremstyle{definition}
\newtheorem{conjecture}[theorem]{Conjecture}

\theoremstyle{remark}
\newtheorem{remark}[theorem]{Remark}

\newcommand{\GL}{\operatorname{GL}}

\newcommand{\Sp}{\operatorname{Sp}}

\DeclareMathOperator{\Reg}{Reg}

   \DeclareFontFamily{U}{wncy}{}
    \DeclareFontShape{U}{wncy}{m}{n}{<->wncyr10}{}
    \DeclareSymbolFont{mcy}{U}{wncy}{m}{n}
    \DeclareMathSymbol{\Sh}{\mathord}{mcy}{"58}

\makeatletter
\def\@tocline#1#2#3#4#5#6#7{\relax
  \ifnum #1>\c@tocdepth 
  \else
    \par \addpenalty\@secpenalty\addvspace{#2}%
    \begingroup \hyphenpenalty\@M
    \@ifempty{#4}{%
      \@tempdima\csname r@tocindent\number#1\endcsname\relax
    }{%
      \@tempdima#4\relax
    }%
    \parindent\z@ \leftskip#3\relax \advance\leftskip\@tempdima\relax
    \rightskip\@pnumwidth plus4em \parfillskip-\@pnumwidth
    #5\leavevmode\hskip-\@tempdima
      \ifcase #1
       \or\or \hskip 1em \or \hskip 2em \else \hskip 3em \fi%
      #6\nobreak\relax
    \hfill\hbox to\@pnumwidth{\@tocpagenum{#7}}\par
    \nobreak
    \endgroup
  \fi}
\makeatother

\begin{document}

\title[Orthogonal families]
{Selberg orthogonality for half-integral weight modular forms}
\author{Shenghao Hua~\orcidlink{0000-0002-7210-2650}}
\address[1]{Shanghai Institute for Mathematics and Interdisciplinary Sciences (SIMIS), Shanghai, 200433, China}
\address[2]{Research Institute of Intelligent Complex Systems, Fudan University, Shanghai, 200433, China}
\email{huashenghao@vip.qq.com}


\begin{abstract}

The Keating--Snaith conjecture for orthogonal families may be viewed as analogous to a Gaussian distribution with a negative mean, and the possibility that mixed moments resemble a composition of independent moments, these two insights were combined and applied in Lester and Radziwi{\l\l}'s proof of quantum unique ergodicity for half-integral weight automorphic forms, via Soundararajan's method under the Generalized Riemann Hypothesis (GRH). This observation also yields a crucial and nontrivial saving in the resolution of certain arithmetic problems.

Inspired by this, we select a series of typical mixed orthogonal families of $L$-functions: $\mathrm{GL}_2$ quadratic twisted families, Gao and Zhao established a sharp upper bound by building upon Harper's method, and one can replace square-free numbers with primes in this argument.
Under the assumptions of the GRH and the Generalized Ramanujan Conjecture, we present the following three arithmetic applications:

i) The decorrelation of Fourier coefficients of half-integral weight modular forms, specifically, a variant of Selberg orthogonality for distinct half-integral weight modular forms.

ii) The decorrelation of automorphic periods averaged over prime imaginary quadratic fields.

iii) The decorrelation of the analytic orders of isotropy subgroups of Tate--Shafarevich groups of elliptic curves under prime quadratic twists.
\end{abstract}



\maketitle


\section{Fourier coefficients of half-integral weight modular forms}

Let $k \geq 2$ be an integer.
In the theory of modular forms,
cusp forms of half-integral weight, particularly those of weight
$k + \tfrac{1}{2}$ on the congruence
subgroup $\Gamma_0(4N)$, play a central role
in connecting automorphic forms to
deep arithmetic objects.
Among these, the Kohnen plus space
$S^+_{k+1/2}(\Gamma_0(4N))$ is of special interest.
It consists of modular forms $g$ whose
$n$-th Fourier coefficient $c(n)$ vanishes
whenever $(-1)^k n \equiv 2, 3 \pmod{4}$.
This subspace is not only technically
convenient, but also arithmetically
meaningful: in $S^+_{k+1/2}(\Gamma_0(4N))$,
Shimura's correspondence between
half-integral weight forms and integral
weight modular forms is especially
well-behaved.
Kohnen~\cite{Kohnen1980} showed that
$S^+_{k+1/2}(\Gamma_0(4N))$ admits a Hecke algebra
structure and is isomorphic to the space
of level $N$ cusp forms of weight $2k$.
As a result, each Hecke eigenform
$g \in S^+_{k+1/2}(\Gamma_0(4N))$ corresponds to a Hecke eigenform $f$ of weight $2k$ and level $N$, known as the Shimura lift of $g$.
The Fourier coefficients $c(n)$
of such $g$ are known to carry deep
arithmetic information.
A remarkable theorem of Waldspurger,
in an explicit form due to
Kohnen and Zagier~\cite{KohnenZagier1981},
connects these Fourier coefficients
to central values of twisted $L$-functions.
Specifically, for a fundamental discriminant
$d$ with $(-1)^k d > 0$, one has
\begin{equation}\label{eqn:fourier}
|c(|d|)|^2
=
 \frac{(k - 1)!}{\pi^k}
\frac{\langle g, g \rangle}
{\langle f, f \rangle}
L\left( \tfrac{1}{2}, f
\otimes \chi_d \right),
\end{equation}
where $\chi_d$ is the quadratic Dirichlet
character associated to
$\mathbb{Q}(\sqrt{d})$,
and $\langle \cdot, \cdot \rangle$
denotes the Petersson inner product.
This formula shows that the size of
the Fourier coefficient $c(|d|)$ is
governed by the central value of the
twisted $L$-function of the Shimura lift $f$.

We consider products of Fourier coefficients of several half-integral weight modular forms and study the sum of their absolute values over a range of primes. In particular, the summation of the products of Fourier coefficients at prime indices for two distinct forms corresponds to a version of the \emph{Selberg orthogonality problem}. Typically, such orthogonality results are formulated for automorphic representations of $\mathrm{GL}(n)$; for instance, the work of Liu--Wang--Ye~\cite{LWY2005} and the proof of the Rudnick--Sarnak Hypothesis~H~\cite{RS1996} (recently established by Jiang~\cite{Jiang2025}) apply to pairs of self-contragredient $\mathrm{GL}(n)$ representations. These results were later extended by Avdispahi\'c and Smajlovi\'c~\cite{AS2010} to cases where one of the representations is not necessarily self-contragredient. However, half-integral weight modular forms correspond to representations of the metaplectic double cover of $\mathrm{GL}(2)$, a setting not covered by the aforementioned general theory.

\begin{theorem}[Decorrelation of Fourier coefficients of half-integral weight modular forms]\label{thm:sum}
Let $1 \le i \le m$, and let each $g_i$ be a Hecke eigenform in the Kohnen space $S^+_{k_i + 1/2}(\Gamma_0(4N_i))$ with Fourier coefficients $c_i(n)$.
Assume that the forms $g_1, \dots, g_m$ are pairwise distinct.
Let $N_0 = [8, N_1, \dots, N_m]$.
Let $a \bmod{N_0}$ denote a residue class with $a \equiv 1 \pmod{4}$ and $(a, N_0) = 1$.
Assume the GRH holds for
\begin{equation*}
L(s, \mathrm{sym}^2 f_i), \quad
L(s, f_i \times \chi_{\sigma p}), \quad
L(s, \chi_{\sigma p}),
\end{equation*}
with $f_i$ associated with the Shimura lifts of $g_i$, and primes $p \equiv \sigma a \pmod{N_0}$ with $X \le p \le 2X$.
Then we have
\begin{equation*}
\frac{1}{X}
\sum_{\substack{\mathrm{prime }~p \equiv \sigma a \pmod{N_0}\\X \le p\le 2X}}\log p
\prod_{i=1}^{m}
|c_i(p)|
\ll_{g_1,\dots,g_m}
(\log X)^{-\frac{m}{8}}.
\end{equation*}

Specifically, setting $m=2$ yields a variant of Selberg orthogonality for distinct half-integral weight modular forms as
\begin{equation*}
\sum_{\substack{X \le p\le 2X \\ \mathrm{prime }~
p \equiv \sigma a \pmod{N_0}}}
c_1(p)c_2(p)=o(\frac{X}{\log X}).
\end{equation*}
\end{theorem}

\begin{proof}
Combining Theorem~\ref{thm:Lbound} with \eqref{eqn:fourier}.
\end{proof}

\begin{remark}
For Selberg orthogonality, the case of identical forms falls within the standard theory concerning central values of prime quadratic twisted $L$-functions and remains an open problem.
\end{remark}

\begin{remark}
Beyond their magnitudes,
many other properties of the Fourier
coefficients of half-integral weight
modular forms have drawn considerable
attention, most notably the question of
their sign changes; see, for example,
\cite{BruinierKohnen2008,Darreye2020,HeKane2021,HulseKiralKuanLim2012,JLLRW2016,KnoppKohnenPribitkin2003,KohnenLauWu2013,LauRoyerWu2016,LesterRadziwill2021,Xu2023}.
\end{remark}

\section{Orthogonal families}

Selberg~\cite{Selberg1946} demonstrated that on the critical line, the real and imaginary parts of the logarithm of the Riemann zeta function are distributed like real Gaussian variables with mean zero and variance $\frac{1}{2} \log \log T$, as the imaginary part ranges from $T$ to $2T$.
Montgomery~\cite{Montgomery1973} provided a prediction regarding the pair correlation of zeros of the Riemann zeta function on the critical line, describing the expected distribution of their spacings, then Dyson observed that this predicted distribution matches the pair correlation statistics of eigenvalues from the Gaussian Unitary Ensemble in random matrix theory.

Katz and Sarnak~\cite{KatzSarnak1999} extended this framework to families of $L$-functions over function fields, demonstrating that the distribution of their low-lying zeros corresponds to orthogonal, unitary, or symplectic symmetry types, depending on the family. Building on these ideas, Keating and Snaith~\cite{KeatingSnaith2000a,KeatingSnaith2000b} used the Circular Ensemble from random matrix theory to conjecture that the logarithm of the central values of $L$-functions in various families follows different Gaussian distributions, each with explicit leading constants.
This significantly generalized prior results by Conrey--Ghosh~\cite{ConreyGhosh1992,ConreyGhosh1998} and Balasubramanian--Conrey--Heath-Brown~\cite{BCH1985}.
Their framework was subsequently refined and extended by Conrey--Farmer~\cite{ConreyFarmer2000}, Diaconu--Goldfeld--Hoffstein~\cite{DGH2003}, Conrey--Farmer--Keating--Rubinstein--Snaith~\cite{CFKRS2005}.
A corrected and clarified version in the function field setting was later given by Sawin~\cite{Sawin2020}.

Current research on the central values of $L$-functions focuses primarily on three fronts: deriving asymptotic formulas for low-order moments, establishing lower bounds~\cite{RS2006}, and obtaining upper bounds. For the latter, Soundararajan~\cite{Soundararajan2009} developed a method based on the GRH that achieves the predicted order (up to a small power of $\log T$), while Harper's refinement~\cite{Harper2013} removes the extraneous $\varepsilon$-power entirely.

In particular, for orthogonal families of $L$-functions, the moments of central values in the range between order $0$ and $1$ (excluding the endpoints) contribute a negative power of $\log$ in their leading-order asymptotics.
This negative logarithmic contribution has concrete applications in arithmetic problems, as it leads to nontrivial savings in analytic estimates.

Traditionally, studies have focused on twisting a single automorphic form.
Building on Chandee's work~\cite{Chandee2011} on shifted moments of the Riemann zeta function, and later the work of Milinovich and Turnage-Butterbaugh~\cite{MTB2014} on integral moments of product of $L$-functions,
it has been observed that when multiple automorphic forms share the same symmetry type under a common twist, their mixed moments, i.e., products of their $L$-values, exhibit statistical independence under suitable conditions.

The combination of the decorrelation phenomenon observed in mixed moments and the negative logarithmic power bounds predicted by the Keating--Snaith conjecture for low-order moments ($0<k<1$) has emerged as a powerful framework in modern analytic number theory.
This interplay plays a central role in several recent breakthroughs.
For instance, Lester and Radziwi\l\l~\cite{LesterRadziwill2020} proved quantum unique ergodicity for half-integral weight automorphic forms, Huang and Lester~\cite{HuangLester2023} investigated the quantum variance of dihedral Maass forms, while Blomer, Brumley, and Khayutin~\cite{BlomerBrumleyKhayutin2022} proved the joint equidistribution conjecture proposed by Michel and Venkatesh in their 2006 ICM proceedings article~\cite{MichelVenkatesh2006}.
Blomer and Brumley~\cite{BlomerBrumley2024} subsequently proved the joint equidistribution of orbits in arithmetic quotients.
J\"a\"asaari, Lester, and Saha~\cite{JLS2023} established sign changes for coefficients of Siegel cusp forms of degree 2, and showed that the mass of Saito--Kurokawa lifted holomorphic cuspidal Hecke eigenforms for $\Sp_4(\mathbb{Z})$ equidistributes on the Siegel modular variety as the weight tends to infinity~\cite{JLS2024}.
Hua, Huang, and Li~\cite{HuaHuangLi2024} established a case of their joint Gaussian moment conjecture (with the holomorphic version discussed in Huang~\cite{Huang2024}). More recently, Chatzakos, Cherubini, Lester, and Risager~\cite{CCLS2025} obtained a logarithmic improvement on Selberg's longstanding bound for the error term in the hyperbolic circle problem over Heegner points with varying discriminants. Hua~\cite{Hua2025quad} demonstrated that for $0 < p < 2$, the $\ell^p$-norm of some quadratic forms in holomorphic Hecke cusp forms tends to zero asymptotically with respect to expansion in a given orthonormal basis of Hecke eigenforms.

In this paper, we illustrate this phenomenon using quadratic twists of $\mathrm{GL}_2$ forms as an example. The family of quadratic twists of a single $\mathrm{GL}_2$ form constitutes an orthogonal family, and the mixed moments of the central $L$-values of several such forms admit an upper bound consistent with the additive behavior expected under statistical independence.

\section{Quadratic twisted families}

The study of central values of $L$-functions related to quadratic characters is an interesting topic (see e.g.~\cite{BP2022,DGH2003,DW2021,GZ2024a,GH1985,HS2022,HH2023a,
HH2023b,HH2023c,Jutila1981,RS2015,Shen2022,Sono2020,Soundararajan2000,
Soundararajan2008,Soundararajan2021,SY2010,Young2009,Young2013}).
Any real primitive character with the modulus $d$ must be of the form $\chi_d(\cdot)=(\frac{d}{\cdot})$, where $d$ is a fundamental discriminant~\cite[Theorem 9.13]{MV2007}, i.e., a product of pairwise coprime factors of the form $-4$, $\pm 8$ and $(-1)^{\frac{p-1}{2}}p$, where $p$ is an odd prime.
Knowledge about moments of central $L$-values is often the basis for the study of the distribution of central $L$-values.
Jutila~\cite{Jutila1981} and Vinogradov--Takhtadzhyan~\cite{VT1981} independently established asymptotic formulas for the first moment of quadratic Dirichlet $L$-functions, whose conductors are all fundamental discriminants, or all prime numbers.
Afterwards, Goldfeld--Hoffstein~\cite{GH1985} and Young~\cite{Young2009} improved the error term.
There also are many developments in setting up asymptotic formulas of higher moments.
Jutila~\cite{Jutila1981} also obtained an asymptotic formula for the second moment, and Sono~\cite{Sono2020} improved the error term of it.
Soundararajan~\cite{Soundararajan2000} got an asymptotic formula for the third moment, and he proved at least $87.5\%$ of the odd square-free integers $d\geq 0$, $L(\frac{1}{2},\chi_{8d})\neq 0$, whereas $L(\frac{1}{2},\chi)\neq 0$ holds for all primitive quadratic characters is a well-known conjecture of Chowla~\cite{Chowla1965}.
Diaconu--Goldfeld--Hoffstein~\cite{DGH2003}, Young~\cite{Young2013}, and Diaconu--Whitehead~\cite{DW2021} improved the error term for the third moment.

The study of the distribution of central values of automorphic $L$-functions in quadratic twists families is also active.
Radziwi{\l \l} and Soundararajan~\cite{RS2015} established an asymptotic formula for the first moment of central $L$-values
of quadratic twists of elliptic curves, and they used this result to derive a one-sided central limit theorem and unconditional upper bounds of small moments corresponding to Keating--Snaith conjecture.
Recently they established a two-sided central limit theorem under the GRH~\cite{RS2024}.
Shen~\cite{Shen2022} set up asymptotic formulas with small error term for the first moment of quadratic twists of modular $L$-functions and the first derivative of quadratic twists of modular $L$-functions.
Recently, Li~\cite{Li2022} got an asymptotic formula for the second moment of quadratic twists of modular $L$-functions, which was previously known conditionally on the Generalized Riemann Hypothesis by the work of Soundararajan and Young~\cite{SY2010}.
Combining Soundararajan's third moment method~\cite{Soundararajan2000}, Young's recursion method~\cite{Young2009} and Heath-Brown's large sieve for quadratic characters~\cite{HB1995}, Hua and Huang~\cite{HH2023a} established asymptotic formulas for the twisted first moment of twisted $\GL_3$ $L$-functions, and they employed these to determine the correspondence for the original Hecke--Maass cusp form or its dual form by central $L$-values.
Also recently, Gao and Zhao~\cite{GZ2024a} built an asymptotic formula for the twisted first moment of central values of the product of a quadratic Dirichlet $L$-function and a quadratic twist of a modular $L$-function.

Hoffstein and Lockhart~\cite{HL1999} extended Heath-Brown's idea to prove an extreme central values result for quadratic twists of modular $L$-functions.
Recently, Hua and Huang~\cite{HH2023c} got an extreme central $L$-values result for almost prime quadratic twists of an elliptic curve $E$, motivated by studying how large the extreme orders of the Tate--Shafarevich groups in the quadratic twist family of an elliptic curve are under the Birth--Swinnerton-Dyer conjecture. The restriction ``almost prime quadratic twists" is just used to control the size of the Tamagawa numbers.
Tamagawa number is bounded by the square of the divisor function trivially, which is much larger than both the conjectured order of magnititude of the extrema of the $L$-function and the order of magnititudethe of the extremum of the $L$-function that can be detected at present.

For the mixed moments of the central values of quadratic twisted $\GL_2$ $L$-functions associated to multiple forms, we have the following conditional result.

\begin{theorem}[Decorrelation of prime quadratic twisted $L$-functions]\label{thm:Lbound}
Let $1 \le i \le m$, and suppose each $f_i$ is either an even-weight Hecke eigenform or a Hecke--Maass form satisfying the Generalized Ramanujan Conjecture (GRC), with level $N_i$.
Assume the forms $f_1, \dots, f_m$ are pairwise distinct.
Fix $\sigma=\pm 1$.
Let $N_0 = [8, N_1, \dots, N_m]$.
Let $a \bmod{N_0}$ denote a residue class with $a \equiv 1 \pmod{4}$ and $(a, N_0) = 1$.
Assume the GRH holds for
\begin{equation*}
L(s, \mathrm{sym}^2 f_i), \quad
L(s, f_i \times \chi_{\sigma p}), \quad
L(s, \chi_{\sigma p})
\end{equation*}
for all $f_i$ and primes $p \equiv \sigma a \pmod{N_0}$ with $X \le p \le 2X$.
Then for any $\ell_1, \dots, \ell_m > 0$, we have
\begin{equation*}
\frac{1}{X}
\sum_{\substack{X \le p\le 2X \\ \mathrm{prime }~
p \equiv \sigma a \pmod{N_0}}}\log p
\prod_{i=1}^{m}
L\left(\tfrac{1}{2}, f_i \times \chi_{\sigma p}\right)^{\ell_i}
\ll_{f_1, \dots, f_m, \ell_1, \dots, \ell_m}
(\log X)^{\sum_{i=1}^{m} \frac{\ell_i(\ell_i - 1)}{2} }
\end{equation*}
as $X \to \infty$.
\end{theorem}

\begin{remark}
  For each $1 \le i \le m$, let $f_i$ have root number $\epsilon_{f_i}$.
Let $a$ be an integer such that, for any prime $p$ with $p\equiv \sigma a \pmod{N_0}$, the root numbers satisfy
\begin{equation*}
\epsilon_{f_i}(\sigma p) = \epsilon_{f_i} \chi_{\sigma p}(-N_i) = 1, \quad \text{for all } i = 1, \dots, m.
\end{equation*}
If this condition fails, then at least one of the $m$ associated $L$-functions vanishes identically.
\end{remark}

\begin{proof}
First adapt Lemma 2.8 of Gao--Zhao~\cite{GZ2023} to a congruence version, use it in place of Lemma 2.8 in Gao--Zhao~\cite{GZ2024b}, and then proceed as in the proofs of Theorem 1.2 and Corollary 1.4 of the same work~\cite{GZ2024b}.
\end{proof}

%

\section{Automorphic periods}

Let \(q\) be a non-degenerate integral ternary quadratic form, and let \( \mathrm{SO}_q\) be its special orthogonal group.
There exists a quaternion algebra \(B/\mathbb{Q}\) with reduced norm \(\mathrm{Nr}_B\) and a scalar \(\lambda \in \mathbb{Q}^\times\) such that
\[
(\mathbb{Q}^3, q) \simeq (B_0, \lambda \mathrm{Nr}_B),
\]
where \(B_0\) is the subspace of trace-zero quaternions. This identification allows us to replace \(\mathbb{Z}^3\) by a lattice \(L \subset B_0\) with quadratic form \(q = \lambda \mathrm{Nr}_B\).

Under this isometry, \(\mathrm{SO}_q \simeq G := \mathrm{PB}^\times = B^\times/\mathbb{Q}^\times\), acting on \(B_0\) by conjugation. The genus classes correspond to
\[
G(\mathbb{Q}) \backslash G(\mathbb{A}_f) / K_f,
\]
where \(K_f\) stabilizes \(L \otimes \widehat{\mathbb{Z}}\).

For each positive integer \(d\),
Representations \((a,b,c)\) of \(d\) by \(q\) correspond to maximal tori
\[
T \simeq \mathrm{Res}_{K/\mathbb{Q}} \mathbb{G}_m / \mathbb{G}_m,
\]
where \(K = \mathbb{Q}(\sqrt{-d/\lambda})\). These tori give rise to adelic torus orbits
\[
[Tg] := T(\mathbb{Q}) \backslash T(\mathbb{A}) \cdot g \subset [G] := G(\mathbb{Q}) \backslash G(\mathbb{A}),
\]
which encode arithmetic information on representations of \(d\) by \(q\).

The equidistribution of these orbits comes from Duke's theorem~\cite{Duke1988} and reduces to studying Weyl sums (automorphic periods)
\[
P_d(\varphi) := \int_{[T]} \varphi(tg) \, dt,
\]
for automorphic forms \(\varphi\) on \([G]\).

Waldspurger's formula relates the square of these toric periods twisted by characters \(\chi\) to central values of automorphic \(L\)-functions:
\begin{equation}\label{eqn:Waldspurgerauto}
\frac{\left|\int_{[T]} \varphi(tg) \chi(t) dt \right|^2 }{\prod_v \langle \varphi_v, \varphi_v \rangle} = c|d|^{-\frac{1}{2}} L\bigl(\frac{1}{2},\pi^{\mathrm{JL}}_K \otimes \chi, \frac{1}{2}\bigr)
L\bigl(1,\pi^{\mathrm{JL}}_K \otimes \chi\bigr)^{-1}
\prod_v \mathcal{P}_v(\varphi_v, \chi_v),
\end{equation}
where \(\pi^{\mathrm{JL}}\) is the Jacquet--Langlands transfer of \(\pi\) to \(\mathrm{PGL}_2\), \(\pi^{\mathrm{JL}}_K\) its base change to \(K\), and \(\mathcal{P}_v\) are local period integrals.

Without loss of generality, let $\lambda=1$.
For quadratic characters \(\chi_d\) associated to \(K = \mathbb{Q}(\sqrt{-d})\), these periods encode the central values of quadratic twist \(L\)-functions
\[
L\bigl(s, \pi^{\mathrm{JL}} \times \chi_{-d}\bigr).
\]

Thus, the geometric automorphic periods over special torus orbits parametrized by discriminants \(d\) are intimately linked to the analytic behavior of quadratic twist \(L\)-functions for \(\mathrm{GL}_2\) automorphic forms.
Then we have the following decorrelation for the automorphic periods at special points.

\begin{theorem}[Decorrelation of automorphic periods]
Let $1 \le i \le m$, and suppose each $\varphi_i$ is either an even-weight Hecke eigenform or a Hecke--Maass form satisfying the GRC, with level $N_i$.
Let $N_0 = [8, N_1, \dots, N_m]$.
Let $a \bmod{N_0}$ denote a residue class with $a \equiv 1 \pmod{4}$ and $(a, N_0) = 1$.
Assume the GRH holds for
\begin{equation*}
L(s, \mathrm{sym}^2 \varphi_i), \quad
L(s, \varphi_i \times \chi_{-p}), \quad
L(s, \chi_{-p})
\end{equation*}
for all $f_i$ and for all primes $p \equiv -a \pmod{N_0}$ with $X \le p \le 2X$.
Then we have
\[
X^{\frac{m}{4}-1} \sum_{\substack{K=\mathbb{Q}(\sqrt{-p}) \\\mathrm{prime }~ p \equiv -a \pmod{N_0} \\ X \leq p \leq 2X}} \log p
\prod_{i=1}^m |P_{-p}(\varphi_i)| \ll_{\varphi_1,\dots,\varphi_m} (\log X)^{-\frac{m}{8}}.
\]
\end{theorem}

\begin{proof}
Combining Theorem~\ref{thm:Lbound} with \eqref{eqn:Waldspurgerauto}.
\end{proof}

\section{Isotropy subgroups of Tate--Shafarevich groups of elliptic curves}

For an elliptic curve \( E/\mathbb{Q} \), we normalize its \( L \)-function so that its central value is at \( s = \frac{1}{2} \).
The Birch and Swinnerton-Dyer (BSD) conjecture relates the first non-zero term in the Taylor expansion of the \(L\)-function at the central point to various arithmetic quantities, including the order of the Shafarevich--Tate group, the Tamagawa numbers, and the regulator.

\begin{conjecture}[BSD conjecture]
For $E$ over $\mathbb{Q}$,
\begin{equation}\label{eqn:BSD}
  \lim_{s \to \frac{1}{2}} \frac{L(s,E)}{(s - \frac{1}{2})^r} =
\frac{ \#\Sh(E) \Reg(E) \Omega_{E} \prod_{p \text{ prime}} c_p(E) }
{ \left( \#E_{\text{tors}} \right)^2 },
\end{equation}
where:
\begin{itemize}
  \item \( r = r(E)=\text{ord}_{s=\frac{1}{2}} L(E, s) \) is the vanishing order of $L$-function at the central point, i.e., the analytic rank of \( E \),
  \item \( \Sh(E) \) is the Shafarevich--Tate group,
  \item $\Reg(E)$ is the regulator which is defined as the determinant of the N\'eron--Tate height pairing on the free part of $E(\mathbb{Q})$,
  \item \( \Omega_{E} \) is the real period,
  \item \( c_p(E) \) are the Tamagawa numbers at the finite places \( p \) of $\mathbb{Q}$,
  \item \( E_{\text{tors}} \) is the torsion subgroup of \( E(\mathbb{Q}) \).
\end{itemize}
\end{conjecture}

For an elliptic curve over $\mathbb{Q}$, when the analytic rank, is at most 1, then the results of Gross--Zagier~\cite{GrossZagier1986} and Kolyvagin~\cite{Kolyvagin1990} imply that the analytic rank coincides with the algebraic rank.
Therefore, for \( L(\frac{1}{2},E) \neq 0\), we have $r=0$ and $\Reg(E)=1$.
According to Goldfeld's conjecture~\cite{Goldfeld1979}, $50\%$ of the quadratic twists $E^{(d)}$ of $E$ satisfy this condition.

Cassels constructed an alternating bilinear pairing
\[
\langle \cdot, \cdot \rangle : \Sh(E) \times \Sh(E) \to \mathbb{Q}/\mathbb{Z},
\]
on the Tate--Shafarevich group \(\Sh(E)\), which is non-degenerate assuming \(\Sh(E)\) is finite.
This non-degeneracy implies that \(\Sh(E)\) is a finite self-dual group under the pairing.
Consequently, the order of \(\Sh(E)\) must be a perfect square, since a non-degenerate alternating pairing can only exist on a finite abelian group whose order is a perfect square (or on a vector space of even dimension). Moreover, there exists an isotropic subgroup
\[
H(E) \subseteq \Sh(E),
\]
meaning that \(\langle x, y \rangle = 0\) for all \( x, y \in H(E) \), with $\#H(E)^2=\# \Sh(E)$, and such subgroups generate \(\Sh(E)\).

Under the assumptions of the BSD conjecture and the GRH, one can observe a decorrelation between the isotropic subgroups of the Tate--Shafarevich groups of different elliptic curves when considering their quadratic twists of the same rank zero.

\begin{theorem}[Decorrelation of analytic orders of isotropy subgroups of Tate--Shafarevich groups of elliptic curves]
Let $1 \le i \le m$, and suppose each conductor $N_i$ elliptic curve $E_i$ is associate to a weight two Hecke eigenform $f_i$ with root number $\epsilon_{f_i}$.
Fix $\sigma = \pm 1$.
Let $N_0 = [8, N_1, \dots, N_m]$.
Let $a \bmod{N_0}$ denote a residue class with $a \equiv 1 \pmod{4}$ and $(a, N_0) = 1$, and for primes $p$ with $p \equiv
\sigma a \pmod{N_0}$, the root numbers satisfy
\begin{equation*}
\epsilon_{f_i}(\sigma p)= \epsilon_{f_i} \chi_{\sigma p}(-N_i) = 1, \quad \text{for all } i = 1, \dots, m.
\end{equation*}

Assume the GRH holds for
\begin{equation*}
L(s, \mathrm{sym}^2 f_i), \quad
L(s, f_i \times \chi_{\sigma p}), \quad
L(s, \chi_{\sigma p})
\end{equation*}
for all $f_i$ and primes $p \equiv \sigma a \pmod{N_0}$ with $X \le p \le 2X$, and the BSD conjecture holds for $E_i^{(p)}$.
Then we have
\[
\frac{1}{X}
\sum_{\substack{X \le p\le 2X \\ \mathrm{prime }~
p \equiv \sigma a \pmod{N_0}}}\log p
\prod_{i=1}^{m}
\frac{\delta_{r(E_i^{(\sigma p)}=0)}\# H(E_i^{(\sigma p)})}{p^{1/4}}
\ll_{E_1,\dots,E_m} (\log X)^{-\frac{m}{8}}.
\]
\end{theorem}

\begin{proof}
  Mazur~\cite{Mazur1977} showed that for $E$ over $\mathbb{Q}$, $1 \leq |E_{\mathrm{tors}}| \leq 12$.
From Pal \cite[Proposition 2.5]{Pal2012}, we know that for square-free integers \( d \) satisfying \(\gcd(|d|, 2N) = 1\), the real period \(\Omega(E^{(d)})\) satisfies
\[
\Omega(E^{(d)}) = \tilde{u} \sqrt{|d|} \, \Omega(E),
\]
where \(\tilde{u} \in \tfrac{1}{2}\mathbb{Z}\) is a constant depending only on \(E\) and \(d\).
Combining Theorem~\ref{thm:Lbound} with the explicit BSD formula \eqref{eqn:BSD}.
\end{proof}

\begin{remark}
Using sieve methods as in Hua and Huang~\cite{HH2023c}, for integers \( d \) with a bounded number of prime factors, the Tamagawa numbers of \( E^{(d)} \) are bounded. In this case, for most twists, the order of the Tate--Shafarevich group can grow as large as \(\sqrt{|d|}\). Therefore, it is reasonable to use \(|d|^{1/4}\) as a comparison scale for \(\# H(E^{(d)})\).
\end{remark}

\section*{Acknowledgements}

The author gratefully acknowledge the helpful discussion with Xinchen Miao.



\begin{thebibliography}{10}

\bibitem{AS2010}
Avdispahi\'c, Muharem; Smajlovi\'c, Lejla.
On the Selberg orthogonality for automorphic $L$-functions.
\emph{Arch. Math. (Basel)} 94 (2010), no. 2, 147--154.


\bibitem{BCH1985}
Balasubramanian, Ramachandran; Conrey, John Brian; Heath-Brown, Roger.
Asymptotic mean square of the product of the Riemann zeta-function and a Dirichlet polynomial.
\emph{J. Reine Angew. Math.} 357 (1985), 161--181.


\bibitem{BP2022}
Baluyot, Siegfred; Pratt, Kyle.
Dirichlet $L$-functions of quadratic characters of prime conductor at the central point.
\emph{J. Eur. Math. Soc. (JEMS)} 24 (2022), no. 2, 369--460.


\bibitem{BlomerBrumleyKhayutin2022}
Blomer, Valentin; Brumley, Farrell; Khayutin, Ilya.
The mixing conjecture under GRH.
(Preprint).
arXiv: 2212.06280.


\bibitem{BlomerBrumley2024}
Blomer, Valentin; Brumley, Farrell. Simultaneous equidistribution of toric periods and fractional moments of $L$-functions.
\emph{J. Eur. Math. Soc. (JEMS)} 26 (2024), no. 8, 2745--2796.


\bibitem{BruinierKohnen2008}
Bruinier, Jan Hendrik; Kohnen, Winfried.
Sign changes of coefficients of half
integral weight modular forms.
\emph{Modular forms on Schiermonnikoog,} 57--65,
\emph{Cambridge Univ. Press, Cambridge,} 2008.


\bibitem{Chandee2009}
Chandee, Vorrapan.
Explicit upper bounds for $L$-functions on the critical line.
\emph{Proc. Amer. Math. Soc.} 137(2009), no. 12, 4049--4063.


\bibitem{Chandee2011}
Chandee, Vorrapan.
On the correlation of shifted values of the Riemann zeta function.
\emph{Q. J. Math.} 62 (2011), no. 3, 545--572.


\bibitem{ChandeeSoundararajan2011}
Chandee, Vorrapan; Soundararajan, Kannan.
Bounding $|\zeta(\frac{1}{2}+it)|$ on the Riemann hypothesis.
\emph{Bull. Lond. Math. Soc.} 43 (2011), no. 2, 243--250.


\bibitem{CCLS2025}
Chatzakos, Dimitrios; Cherubini, Giacomo; Lester, Stephen; Risager, Morten S.
The hyperbolic circle problem over Heegner points.
(Preprint).
arXiv: 2506.13883.


\bibitem{Chowla1965}
Chowla, Sarvadaman.
The Riemann hypothesis and Hilbert's tenth problem.
\emph{Norske Vid. Selsk. Forh. (Trondheim) } 38 (1965), 62--64.


\bibitem{ConreyFarmer2000}
Conrey, John Brian; Farmer, David W.
Mean values of $L$-functions and symmetry.
\emph{Internat. Math. Res. Notices} 2000, no. 17, 883--908.


\bibitem{CFKRS2005}
Conrey, John Brian; Farmer, David W.; Keating, Jonathan P.; Rubinstein, M. O.; Snaith, Nina C.
Integral moments of $L$-functions. Proc. London Math. Soc. (3) 91 (2005), no. 1, 33--104.


\bibitem{ConreyGhosh1992}
Conrey, John Brian; Ghosh, Amit.
Mean values of the Riemann zeta-function. III.
\emph{Proceedings of the Amalfi Conference on Analytic Number Theory (Maiori, 1989),} 35--59, \emph{Univ. Salerno, Salerno,} 1992.


\bibitem{ConreyGhosh1998}
Conrey, John Brian; Ghosh, Amit.
A conjecture for the sixth power moment of the Riemann zeta-function.
\emph{Internat. Math. Res. Notices} 1998, no. 15, 775--780.


\bibitem{Darreye2020}
Darreye, Corentin.
Sign of Fourier coefficients of
half-integral weight modular forms in
arithmetic progressions.
\emph{Res. Number Theory} 6 (2020), no. 4, Paper No. 46.


\bibitem{DGH2003}
Diaconu, Adrian; Goldfeld, Dorian; Hoffstein, Jeffrey.
Multiple Dirichlet series and moments of zeta and $L$-functions. \emph{Compositio Math.} 139 (2003), no. 3, 297--360.


\bibitem{DW2021}
Diaconu, Adrian; Whitehead, Ian.
On the third moment of $L(\frac{1}{2},\chi_d)$ II: the number field case.
\emph{J. Eur. Math. Soc. (JEMS)} 23 (2021), no. 6, 2051--2070.


\bibitem{Duke1988}
Duke, William Drexel.
Hyperbolic distribution problems and half-integral weight Maass forms.
\emph{Invent. Math.} 92 (1988), no. 1, 73--90.


\bibitem{GZ2023}
Gao, Peng; Zhao, Liangyi.
Bounds for moments of quadratic Dirichlet $L$-functions of prime-related moduli. \emph{Colloq. Math.} 171 (2023), no. 1, 61--77.


\bibitem{GZ2024a}
Gao, Peng; Zhao, Liangyi.
Twisted first moment of quadratic and quadratic twist $L$-functions.
\emph{Funct. Approx. Comment. Math.}
70 (2024), no. 1, 101--127.


\bibitem{GZ2024b}
Gao, Peng; Zhao, Liangyi.
Bounds for moments of twisted Fourier coefficients of modular forms.
(Preprint).
arXiv: 2412.12515.


\bibitem{GelbartJacquet1978}
Gelbart, Stephen; Jacquet, Herv\'{e}.
A relation between automorphic representations of $\GL(2)$ and $\GL(3)$.
\emph{Ann. Sci. \'{E}cole Norm. Sup.} (4) 11 (1978), no. 4, 471--542.


\bibitem{Goldfeld1979}
Goldfeld, Dorian.
Conjectures on elliptic curves over quadratic fields.
\emph{Number theory, Carbondale 1979 (Proc. Southern Illinois Conf., Southern Illinois Univ., Carbondale, Ill., 1979),} pp. 108--118, Lecture Notes in Math., 751, \emph{Springer, Berlin,} 1979.


\bibitem{GH1985}
Goldfeld, Dorian; Hoffstein, Jeffrey.
Eisenstein series of $\frac{1}{2}$-integral weight and the mean value of real Dirichlet $L$-series.
\emph{Invent. Math.} 80 (1985), no. 2, 185--208.


\bibitem{GrossZagier1986}
Gross, Benedict H.; Zagier, Don B. Heegner points and derivatives of $L$-series.
\emph{Invent. Math.} 84 (1986), no. 2, 225--320.


\bibitem{Hagen2024}
Hagen, Markus Val{\aa}s.
Sharp conditional moment bounds for products of $L$-functions.
(Preprint).
arXiv: 2409.19780.


\bibitem{Harper2013}
Harper, Adam J.
Sharp conditional bounds for moments of the Riemann zeta function.
(Preprint).
arXiv: 1305.4618.


\bibitem{HeKane2021}
He, Zilong; Kane, Ben.
Sign changes of Fourier coefficients of
cusp forms of half-integral weight
over split and inert primes in quadratic
number fields.
\emph{Res. Number Theory} 7 (2021), no. 1, Paper No. 10.


\bibitem{HS2022}
Heap, Winston; Soundararajan, Kannan.
Lower bounds for moments of zeta and $L$-functions revisited.
\emph{Mathematika} 68 (2022), no. 1, 1--14.


\bibitem{HB1995}
Heath-Brown, Roger.
A mean value estimate for real character sums.
\emph{Acta Arith.} 72 (1995), no. 3, 235--275.


\bibitem{HL1999}
Hoffstein, Jeffrey; Lockhart, Paul.
Omega results for automorphic $L$-functions. \emph{Automorphic forms, automorphic representations, and arithmetic (Fort Worth, TX, 1996),} 239--250, Proc. Sympos. Pure Math., 66, Part 2, \emph{Amer. Math. Soc., Providence, RI,} 1999.


\bibitem{HH2023a}
Hua, Shenghao; Huang, Bingrong.
Determination of $\GL(3)$ cusp forms by central values of quadratic twisted  $L$-functions.
\emph{Int. Math. Res. Not. IMRN} 2023, no. 9, 7976--8007.


\bibitem{HH2023b}
Hua, Shenghao; Huang, Bingrong.
Lower bounds for moments of quadratic twisted self-dual $\GL(3)$ central $L$-values.
\emph{Acta Math. Sin. (Engl. Ser.)} 39 (2023), no. 11, 2139--2148.


\bibitem{HH2023c}
Hua, Shenghao; Huang, Bingrong.
Extreme central $L$-values of almost prime quadratic twists of elliptic curves.
\emph{Sci. China Math.} 66 (2023), no. 12, 2755--2766.


\bibitem{HuaHuangLi2024}
Hua, Shenghao; Huang, Bingrong; Li, Liangxun.
Joint value distribution of Hecke--Maass forms.
(Preprint).
arXiv: 2405.00996.


\bibitem{Hua2025quad}
Hua, Shenghao.
Quadratic forms of holomorphic cusp forms
and the decay of their $\ell^p$-norms for
$0 < p < 2$.
(Preprint).
arXiv: 2507.21951.


\bibitem{HuangLester2023}
Huang, Bingrong; Lester, Stephen.
Quantum variance for dihedral Maass forms. \emph{Trans. Amer. Math. Soc.} 376 (2023), no. 1, 643--695.


\bibitem{Huang2024}
Huang, Bingrong.
Joint distribution of Hecke eigenforms.
(Preprint).
arXiv: 2406.03073.


\bibitem{HulseKiralKuanLim2012}
Hulse, Thomas A.; Kiral, E. Mehmet;
Kuan, Chan Ieong; Lim, Li-Mei.
The sign of Fourier coefficients of
half-integral weight cusp forms.
\emph{Int. J. Number Theory} 8 (2012), no. 3,
749--762.


\bibitem{JLS2023}
J\"a\"asaari, Jesse; Lester, Stephen; Saha, Abhishek.
On fundamental Fourier coefficients of Siegel cusp forms of degree 2.
\emph{J. Inst. Math. Jussieu} 22 (2023), no. 4, 1819--1869.


\bibitem{JLS2024}
J\"a\"asaari, Jesse; Lester, Stephen; Saha, Abhishek.
Mass equidistribution for Saito--Kurokawa lifts.
\emph{Geom. Funct. Anal.} 34 (2024), no. 5, 1460--1532.


\bibitem{JLLRW2016}
Jiang, Yujiao; Lau, Yuk-kam; L\"u, Guangshi;
Royer, Emmanuel; Wu, Jie.
Sign changes of fourier coefficients of
modular forms of half integral weight, 2.
(Preprint).
arXiv: 1602.08922.


\bibitem{Jiang2025}
Jiang, Yujiao.
On Hypothesis H of Rudnick and Sarnak.
(Preprint).
arXiv: 2507.20653.


\bibitem{Jutila1981}
Jutila, Matti Ilmari.
On the mean value of $L(\frac{1}{2},\chi)$ for real characters.
\emph{Analysis} 1 (1981), no. 2, 149--161.


\bibitem{KatzSarnak1999}
Katz, Nicholas M.; Sarnak, Peter.
Random matrices, Frobenius eigenvalues, and monodromy. American Mathematical Society Colloquium Publications, 45.
\emph{American Mathematical Society, Providence, RI,} 1999.


\bibitem{KeatingSnaith2000a}
Keating, Jonathan P.; Snaith, Nina C.
Random matrix theory and $\zeta(1/2+it)$. \emph{Comm. Math. Phys.} 214 (2000), no. 1, 57--89.


\bibitem{KeatingSnaith2000b}
Keating, Jonathan P.; Snaith, Nina C.
Random matrix theory and $L$-functions at $s=1/2$.
\emph{Comm. Math. Phys.} 214 (2000), no. 1, 91--110.


\bibitem{KnoppKohnenPribitkin2003}
Knopp, Marvin; Kohnen, Winfried;
Pribitkin, Wladimir.
On the signs of Fourier coefficients of
cusp forms.
Rankin memorial issues.
\emph{Ramanujan J.} 7 (2003), no. 1-3, 269--277.


\bibitem{Kohnen1980}
Kohnen, Winfried.
Modular forms of half-integral weight
on $\Gamma_0(4)$.
\emph{Math. Ann.} 248 (1980), no. 3, 249--266.


\bibitem{KohnenLauWu2013}
Kohnen, Winfried; Lau, Yuk-Kam;
Wu, Jie.
Fourier coefficients of cusp forms
of half-integral weight.
\emph{Math. Z.} 273 (2013), no. 1-2, 29--41.


\bibitem{KohnenZagier1981}
Kohnen, Winfried; Zagier, Don Bernard.
Values of $L$-series of modular forms
at the center of the critical strip.
\emph{Invent. Math.} 64 (1981), no. 2,
175--198.


\bibitem{Kolyvagin1990}
Kolyvagin, Victor A.
Euler systems.
\emph{The Grothendieck Festschrift, Vol. II,} 435--483, Progr. Math., 87, \emph{Birkh\"auser Boston, Boston, MA,} 1990.


\bibitem{LauRoyerWu2016}
Lau, Yuk-kam; Royer, Emmanuel; Wu, Jie.
Sign of Fourier coefficients of modular
forms of half-integral weight.
\emph{Mathematika} 62 (2016), no. 3, 866--883.


\bibitem{LesterRadziwill2020}
Lester, Stephen; Radziwi{\l \l}, Maksym. Quantum unique ergodicity for half-integral weight automorphic forms.
\emph{Duke Math. J.} 169 (2020), no. 2, 279--351.


\bibitem{LesterRadziwill2021}
Lester, Stephen; Radziwi{\l \l}, Maksym.
Signs of Fourier coefficients of
half-integral weight modular forms.
\emph{Math. Ann.} 379 (2021), no. 3-4, 1553--1604.


\bibitem{Li2022}
Li, Xiannan.
Moments of quadratic twists of modular $L$-functions.
\emph{Invent. Math.} 237 (2024), no. 2, 697--733.


\bibitem{LWY2005}
Liu, Jianya; Wang, Yonghui; Ye, Yangbo.
A proof of Selberg's orthogonality for automorphic $L$-functions.
\emph{Manuscripta Math.} 118 (2005), no. 2, 135--149.


\bibitem{Mazur1977}
Mazur, Barry C.
Modular curves and the Eisenstein ideal. With an appendix by Mazur and M. Rapoport. \emph{Inst. Hautes \'Etudes Sci. Publ. Math.} No. 47 (1977), 33--186.


\bibitem{MichelVenkatesh2006}
Michel, Philippe; Venkatesh, Akshay. Equidistribution, $L$-functions and ergodic theory: on some problems of Yu. Linnik. \emph{International Congress of Mathematicians. Vol. II,} 421--457, \emph{Eur. Math. Soc., Z\"{u}rich,} 2006.


\bibitem{MTB2014}
Milinovich, Micah B.; Turnage-Butterbaugh, Caroline L. Moments of products of automorphic $L$-functions.
J. Number Theory 139 (2014), 175--204.


\bibitem{Montgomery1973}
Montgomery, Hugh. L.
The pair correlation of zeros of the zeta function.
\emph{Analytic number theory (Proc. Sympos. Pure Math., Vol. XXIV, St. Louis Univ., St. Louis, Mo., 1972),} pp. 181--193, Proc. Sympos. Pure Math., Vol. XXIV, \emph{Amer. Math. Soc., Providence, RI,} 1973.


\bibitem{MV2007}
Montgomery, Hugh L.; Vaughan, Robert C.
Multiplicative number theory. I. Classical theory.
Cambridge Studies in Advanced Mathematics, 97. \emph{Cambridge University Press, Cambridge,} 2007.


\bibitem{Pal2012}
Pal, Vivek.
Periods of quadratic twists of elliptic curves. With an appendix by Amod Agashe.
\emph{Proc. Amer. Math. Soc.} 140 (2012), no. 5, 1513--1525.


\bibitem{RS2015}
Radziwi{\l \l}, Maksym; Soundararajan, Kannan.
Moments and distribution of central $L$-values of quadratic twists of elliptic curves.
\emph{Invent. Math.} 202 (2015), no. 3, 1029--1068.


\bibitem{RS2024}
Radziwi{\l \l}, Maksym; Soundararajan, Kannan.
Conditional lower bounds on the distribution of central values in families of $L$-functions.
\emph{Acta Arith.} 214 (2024), 481--497.


\bibitem{Ramakrishnan2014}
Ramakrishnan, Dinakar.
An exercise concerning the selfdual cusp forms on $\GL(3)$.
\emph{Indian J. Pure Appl. Math.} 45 (2014), no. 5, 777--785.


\bibitem{RS1996}
Rudnick, Ze\'ev; Sarnak, Peter.
Zeros of principal $L$-functions and random matrix theory.
A celebration of John F. Nash, Jr.
\emph{Duke Math. J.} 81 (1996), no. 2, 269--322.


\bibitem{RS2006}
Rudnick, Ze\'ev; Soundararajan, Kannan.
Lower bounds for moments of $L$-functions: symplectic and orthogonal examples.
\emph{Multiple Dirichlet series, automorphic forms, and analytic number theory,} 293--303, Proc. Sympos. Pure Math., 75, \emph{Amer. Math. Soc., Providence, RI,} 2006.


\bibitem{Sawin2020}
Sawin, Will.
A representation theory approach to integral moments of $L$-functions over function fields.
Algebra Number Theory 14 (2020), no. 4, 867--906.


\bibitem{Selberg1946}
Selberg, Atle.
Contributions to the theory of the Riemann zeta-function.
\emph{Arch. Math. Naturvid.} 48 (1946), no. 5, 89--155.


\bibitem{Shen2022}
Shen, Quanli.
The first moment of quadratic twists of modular $L$-functions.
\emph{Acta Arith.} 206 (2022), no. 4, 313--337.


\bibitem{Sono2020}
Sono, Keiju.
The second moment of quadratic Dirichlet $L$-functions.
\emph{J. Number Theory} 206 (2020), 194--230.


\bibitem{Soundararajan2000}
Soundararajan, Kannan.
Nonvanishing of quadratic Dirichlet $L$-functions at $s=1/2$.
\emph{ Ann. of Math. (2)} 152 (2000), no. 2, 447--488.


\bibitem{Soundararajan2008}
Soundararajan, Kannan.
Extreme values of zeta and $L$-functions.
\emph{Math. Ann.} 342 (2008), no. 2, 467--486.


\bibitem{Soundararajan2009}
Soundararajan, Kannan.
Moments of the Riemann zeta function.
\emph{Ann. of Math. (2)} 190 (2) (2009), 981--993.


\bibitem{Soundararajan2021}
Soundararajan, Kannan.
The distribution of values of zeta and $L$-functions.
\emph{ICM--International Congress of Mathematicians. Vol. 2. Plenary lectures,} 1260--1310, \emph{EMS Press, Berlin,} 2023.


\bibitem{SY2010}
Soundararajan, Kannan; Young, Matthew P.
The second moment of quadratic twists of modular $L$-functions.
\emph{J. Eur. Math. Soc. (JEMS)} 12 (2010), no. 5, 1097--1116.


\bibitem{Titchmarsh1986}
Titchmarsh, Edward Charles.
The theory of the Riemann zeta-function.
\emph{The Clarendon Press, Oxford University Press, New York,} 1986.


\bibitem{VT1981}
Vinogradov, A. I.; Takhtadzhyan, L. A. Analogues of the Vinogradov-Gauss formula on the critical line.
Differential geometry, Lie groups and mechanics, IV.
\emph{Zap. Nauchn. Sem. Leningrad. Otdel. Mat. Inst. Steklov. (LOMI)} 109 (1981), 41--82, 180--181, 182--183.


\bibitem{Xu2023}
Xu, Chenran.
Sign changes of coefficients of
half-integral weight Hecke eigenforms.
\emph{Proc. Amer. Math. Soc.} 151 (2023),
no. 11, 4633--4642.


\bibitem{Young2009}
Young, Matthew P.
The first moment of quadratic Dirichlet $L$-functions.
\emph{Acta Arith.} 138 (2009), no. 1, 73--99.


\bibitem{Young2013}
Young, Matthew P.
The third moment of quadratic Dirichlet $L$-functions.
\emph{Selecta Math. (N.S.)} 19 (2013), no. 2, 509--543.

\end{thebibliography}
\end{document}